\documentclass[runningheads,svgnames]{llncs}
\usepackage[T1]{fontenc}
\usepackage{graphicx}
\usepackage{amsmath}
\usepackage{amssymb}
\usepackage{amsfonts}
\usepackage{tikz}
\usepackage{float}
\usepackage{hyperref}
\usepackage{cleveref}
\usepackage{tikz-network}
\usetikzlibrary{positioning, quotes}
\usepackage{comment}
\usepackage{orcidlink}
\usepackage{marvosym}

\spnewtheorem*{notn}{Notation}{\bfseries}{\upshape}
\spnewtheorem*{defn}{Definition}{\bfseries}{\upshape}
\spnewtheorem{clm}{Claim}{\bfseries}{\itshape}
\spnewtheorem{cl}{Claim}{\bfseries}{\itshape}
\spnewtheorem{clam}{Claim}{\bfseries}{\itshape}

\crefname{clm}{Claim}{Claims}
\crefname{cl}{Claim}{Claims}
\crefname{case}{Case}{Cases}

\begin{document}

\title{Bounds on Linear Tur\'{a}n Number for Trees}

\author{
Rajat Adak\inst{1}\textsuperscript{(\Letter)}\orcidlink{0009-0001-2723-5550} \and
Pragya Verma\inst{2}\textsuperscript{(\Letter)}\orcidlink{0009-0003-7725-9424}
}

\authorrunning{R. Adak and P. Verma}

\institute{
Department of Computer Science and Automation,  
Indian Institute of Science, Bangalore, India\\
\email{rajatadak@iisc.ac.in}
\and
Department of Computer Science and Engineering,  
Indian Institute of Technology Bombay, India\\
\email{24d0376@iitb.ac.in}
}
\maketitle
\begin{abstract}

A hypergraph $H$ is said to be \emph{linear} if every pair of vertices lies in at most one hyperedge.
Given a family $\mathcal{F}$ of $r$-uniform hypergraphs, an $r$-uniform hypergraph $H$ is said to be \emph{$\mathcal{F}$-free} if it contains no member of $\mathcal{F}$ as a subhypergraph.
The \emph{linear Tur\'{a}n number} $ex_r^{\mathrm{lin}}(n,\mathcal{F})$ denotes the maximum number of hyperedges in an $\mathcal{F}$-free linear $r$-uniform hypergraph on $n$ vertices.

Gy\'arf\'as, Ruszink\'o, and S\'ark\"ozy~[\emph{Linear Tur\'an numbers of acyclic triple systems}, European J.\ Combin.\ (2022)] initiated the study of bounds on the linear Tur\'an number for acyclic $3$-uniform linear hypergraphs.

In this paper, we extend the study of linear Tur\'{a}n numbers for acyclic systems to higher uniformity.
We first give a construction for linear $r$-uniform trees with $k$ edges that yields the lower bound
$
ex_r^{\mathrm{lin}}(n,T_k^r)\ge {n(k-1)}/{r},
$
under mild divisibility and existence assumptions. 
Next, we study hypertrees with four edges. We prove the exact bound
$
ex_r^{\mathrm{lin}}(n,B_4^r)\le {(r+1)n}/{r}
$
and characterize the extremal hypergraph class, where $B_4^r$ is formed from $S_3^r$ by appending a hyperedge incident to a degree-one vertex.
We also prove the bound
$
ex_r^{\mathrm{lin}}(n,E_4^r)\le {(2r-1)n}/{r}
$
for the crown $E_4^r$. 
Finally, we give a construction showing
$
ex_r^{\mathrm{lin}}(n,P_4^r)\ge {(r+1)n}/{r}
$
under suitable assumptions and conclude with a conjecture on sharp upper bound for $P_4^r$.

\end{abstract}
\section{Introduction}
Extremal graph theory is largely concerned with Tur\'{a}n-type problems, which ask for the maximum number of edges in a graph or hypergraph that avoids a prescribed forbidden substructure. In this work, we focus on a more recent line of research, namely the study of linear Tur\'{a}n numbers. The Tur\'{a}n problem for hypergraphs has been investigated extensively; see, for instance, the surveys \cite{furedi1991turan,furedi2013history,keevash2011hypergraph} and the book \cite{gerbner2018extremal}.

The study of Tur\'{a}n numbers has a long history, beginning with the classical result of Tur\'{a}n (1941) for cliques. Let $ex(n,K_{r+1})$ denote the maximum number of edges in a graph on $n$ vertices that does not contain a copy of $K_{r+1}$ as a subgraph. The Tur\'{a}n graph, denoted by $T(n,r)$, is a complete $r$-partite graph on $n$ vertices whose partite sets are as nearly equal in cardinality as possible (balanced complete $r$-partite).

\begin{theorem}\label{th:Turan}\emph{(Tur\'{a}n \cite{turan1941egy})}
For $n \geq r \geq 1$, $ex(n,K_{r+1}) \leq |E(T(n,r))|$. Equality holds if and only if $G \cong T(n,r)$.
\end{theorem}

In contrast to graphs, where subgraphs are defined in a unique way, hypergraphs admit several natural notions of containment for a given structure. In particular, the same object (such as a path or cycle) can be defined under different interpretations, including Berge, Minimal and Linear notions. This additional flexibility makes Tur\'{a}n-type problems in hypergraphs significantly more complex. The extension of the Tur\'{a}n problem to hypergraphs has been widely studied. Tur\'{a}n-type results for minimal paths and cycles were obtained by Mubayi and Verstra\"{e}te~\cite{mubayi2007minimal}, while exact results for linear paths (in a uniform hypergraph) were established by F\"{u}redi, Jiang, and Seiver~\cite{furedi2014exact}. For Berge paths, tight results were given by Gy\H{o}ri, Katona, and Lemons~\cite{gyHori2016hypergraph}, with a remaining case resolved in~\cite{davoodi2018erdHos}. In this paper, we focus exclusively on the linear framework.
\subsection{Preliminaries}
In a hypergraph $H = (V(H), E(H))$, $V(H)$ denotes the set of vertices, and $E(H)$ is the set of hyperedges in $H$. For $v \in V(H)$, $d(v)$ denotes the degree of the vertex in $H$, that is, the number of hyperedges incident on $v$. Let $\delta(H)$ denote the minimum vertex degree in $H$, similarly let $\Delta(H)$ be the maximum vertex degree.

$H$ is called $r$-uniform if each $e \in E(H)$ contains exactly $r$ vertices. A hypergraph is \textit{linear} when any two distinct hyperedges intersect in at most one vertex.

\subsection{Linear Tur\'{a}n Number}
For a family $\mathcal{F}$ of linear $r$-uniform hypergraphs, the \textit{linear Tur\'{a}n number}, \linebreak $ex_{r}^{\mathrm{lin}}(n,\mathcal{F})$, denotes the maximum number of hyperedges in an $n$-vertex linear $r$-uniform hypergraph that contains no member of $\mathcal{F}$ as a subhypergraph. When $\mathcal{F}=\{F\}$, we abbreviate this notation to $ex_r^{\mathrm{lin}}(n,F)$.

A key construction in this framework is the \textit{expansion} of a graph. Given a graph $G$, its expansion $G^r$ is obtained by inflating each graph edge into an $r$-element hyperedge using $r-2$ new vertices (kept disjoint across edges). The result is automatically a linear $r$-uniform hypergraph.

The term \textit{linear Tur\'{a}n number} was introduced by Collier-Cartaino, Graber, and Jiang~\cite{COLLIER-CARTAINO_GRABER_JIANG_2018}, who established foundational bounds for linear cycles. Yet the phenomenon predates the terminology: the celebrated work of Ruzsa and Szemer\'{e}di~\cite{ruzsa1978triple} can be stated as
$$n^{2-\frac{c}{\sqrt{\log{n}}}}\le ex^{\mathrm{lin}}_{3}(n,C^3_3)=o(n^2),$$
where $c>0$ and $C_3^3$ denotes the $3$-uniform expansion of the $C_3$ ($3$-edge cycle).

The systematic study of acyclic structures in this context was initiated by Gy\'arf\'as, Ruszink\'o, and S\'ark\"ozy~\cite{gyarfas2022linear}, who investigated acyclic triple systems and obtained several bounds for $3$-uniform linear expansions. Zhang and Wang~\cite{zhang2025linear} extended this direction to the $4$-uniform setting and proved corresponding bounds for a variety of forbidden configurations.

The current picture continues to rapidly develop. Zhou and Yuan~\cite{zhou2025turan} derived general $r$-uniform bounds, including for $P_k^r$, showing $ex^{\mathrm{lin}}_r(n,P_k^{r}) \leq \frac{(2r-3)kn}{2}$ for $r\ge3, k\ge4$. Complementing this line, Khormali and Palmer~\cite{khormali2022turan} and Zhang et al.~\cite{zhang2025turan} investigated star-forests, establishing bounds for sufficiently large $n$.

For linear hypertrees on four hyperedges, the setting already becomes rich. Expansions of trees yield several distinct configurations. For $k=2$, the only possibility is $P_2^r$ (where $P_k$ denotes a path on $k$ edges). It is easy to check that $ex_r^{\mathrm{lin}}(n,P_2^r) = \lfloor\frac{n}{r}\rfloor$. For $k=3$, we obtain precisely $P_3^r$ and $S_3^r$ (where $S_k$ denotes a star on $k$ edges). In \cite{zhang2025linear}, it was shown that $ex_r^{\mathrm{lin}}(n,P_3^r) \leq n$. When $k=4$, there are three expansions: $P_4^r$ , $S_4^r$, and $B_4^r$, where $B_4$ is the tree obtained by attaching an extra edge to a leaf of $S_3$.

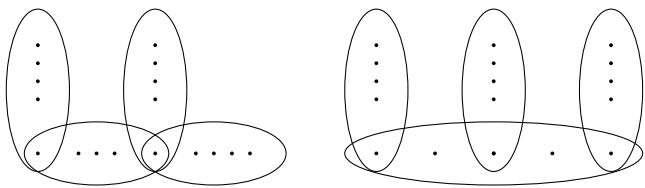
\begin{figure}[ht]
\centering
\begin{tikzpicture}[scale=0.5]

% ================= LEFT FIGURE =================
\begin{scope}[xshift=0cm]

% Ellipses
\draw (0,0) ellipse (0.7 and 1.8);
\draw (2.6,0) ellipse (0.7 and 1.8);
\draw (1.3,-1.4) ellipse (1.6 and 0.7);
\draw (3.9,-1.4) ellipse (1.6 and 0.7);

% Dots in left vertical ellipse
\foreach \y in {1,0.6,0.2,-0.2,-1.4}
  \fill (0,\y) circle (1.2pt);

% Dots in right vertical ellipse
\foreach \y in {1,0.6,0.2,-0.2,-1.4}
  \fill (2.6,\y) circle (1.2pt);

% Dots in bottom-left horizontal ellipse
\foreach \x in {0.9,1.3,1.7}
  \fill (\x,-1.4) circle (1.2pt);

% Dots in bottom-right horizontal ellipse
\foreach \x in {3.5,3.9,4.3,4.7}
  \fill (\x,-1.4) circle (1.2pt);

\end{scope}

% ================= RIGHT FIGURE =================
\begin{scope}[xshift=7.5cm]

% Ellipses
\draw (0,0) ellipse (0.7 and 1.8);
\draw (2.6,0) ellipse (0.7 and 1.8);
\draw (5.2,0) ellipse (0.7 and 1.8);
\draw (2.6,-1.4) ellipse (3.3 and 0.7);

% Dots in vertical ellipses
\foreach \x in {0,2.6,5.2} {
  \foreach \y in {1,0.6,0.2,-0.2,-1.4}
    \fill (\x,\y) circle (1.2pt);
}

% Dots in bottom horizontal ellipse
\foreach \x in {1.3,2.6,3.9}
  \fill (\x,-1.4) circle (1.2pt);

\end{scope}
\end{tikzpicture}
\caption{Configurations of $B_4^r$ and $E_4^r$ (in the figure $r =5$).}
\end{figure}\vspace{-10pt}
Alongside these expansions, there is a notable non-expansion configuration known as the \emph{crown} $E_4^r$, consisting of three pairwise disjoint hyperedges together with a fourth hyperedge, called the \emph{base}, that intersects each of the other three.
Crowns have been studied in a variety of settings in~\cite{carbonero2021crowns,carbonero2022crowns,tang2022turan,zhang2025generalized}. In a related direction, Fletcher~\cite{fletcher2021improved} obtained the bound $ex_3^{\mathrm{lin}}(n,E_4^3)<\frac{5n}{3}$ for the $3$-uniform case. More generally, let $T_k^r$ denote a linear $r$-uniform hypertree on $k$ hyperedges, that is, a connected, \textit{acyclic} $r$-uniform hypergraph in which any two hyperedges intersect in at most one vertex. Here, \textit{acyclic} means that there is no sequence of distinct hyperedges $e_1,\dots,e_t$ ($t\ge3$) with $e_i\cap e_{i+1}\neq\emptyset$ for all $i$ (indices modulo $t$) and all these intersections are distinct. Similar to $E_4^r$, such a hypertree need not arise as an expansion configuration, depending on $k$.

Finding upper bounds on the Tur\'an number of trees is already very difficult even in the graph setting (see the Erd\H{o}s--S\'os conjecture~\cite{erdos1965extremal}, 1963). Therefore, we focus on small hypertrees.

\subsection{Steiner Systems}

A Steiner system $S(t,r,n)$ consists of an $n$-vertex set together with blocks of size $r$ such that every $t$-subset lies in exactly one block. In particular, $S(2,r,n)$ repeatedly emerges in the study of linear Tur\'{a}n extremal problems, as its block count realizes the maximum possible edge density of a linear $r$-uniform hypergraph on $n$ vertices. This connection is not new: Steiner triple systems $S(2,3,n)$ (also denoted by $STS(n)$) have already appeared in the study of linear Tur\'{a}n numbers for triple systems in \cite{gyarfas2022linear,gyarfas2021turan,gyarfas2022linearwicket,sarkozy2023turan}, while Steiner quadruple systems $S(2,4,n)$ (denoted by $SQS(n)$) arise naturally in the $4$-uniform setting, for instance in \cite{zhang2025linear}. These systems serve as extremal or near-extremal constructions in several linear Tur\'{a}n-type problems.

Note that $S(2,r,n)$ may not exist for all values of $r$ and $n$; we return to this point after the following lemma.

\begin{lemma}\label{steinercount}
Let $H$ be a Steiner system $S(2,r,n)$ \emph{(}assuming it exists\emph{)}.
Then
\[
|E(H)|=\frac{\binom{n}{2}}{\binom{r}{2}}=\frac{n(n-1)}{r(r-1)},
\qquad\text{and}\qquad
d(v)=\frac{n-1}{r-1}\ \ \text{for every }v\in V(H).
\]
\end{lemma}

\begin{proof}
Each edge covers $\binom{r}{2}$ distinct vertex-pairs, and by the definition, we get that every pair
among the $\binom{n}{2}$ pairs is covered exactly once. Hence
$|E(H)|=\binom{n}{2}/\binom{r}{2}$.

For the degree: fix $v\in V(H)$. The $(n-1)$ pairs $\{v,u\}$ ($u\neq v$) must each be covered
exactly once. Each edge containing $v$ covers exactly $(r-1)$ such pairs, so
$d(v)=\frac{(n-1)}{(r-1)}$.
\qed\end{proof}

Note that the divisibility conditions implicit in Lemma~\ref{steinercount} are necessary for the existence of a Steiner system $S(2,r,n)$. Indeed, the degree condition $d(v)=\frac{n-1}{r-1}$ requires that $(r-1)\mid(n-1)$, and the edge count $|E(H)|=\frac{\binom{n}{2}}{\binom{r}{2}}$ requires that $\binom{r}{2}\mid\binom{n}{2}$. It is known that these necessary conditions are also sufficient for all sufficiently large $n$~\cite{keevash2014existence}. 
\newline For small parameters, existence is characterized in several cases: $S(2,3,n)$ exists if and only if $n \equiv 1,3 \pmod{6}$, and $S(2,4,n)$ exists if and only if $n \equiv 1,4 \pmod{12}$. Moreover, if $r$ is a prime power, Steiner systems of the form $S(2,r,r^2)$ (corresponding to affine planes of order $r$) are known to exist. In general, however, existence of $S(2,r,r^2)$ is not guaranteed. The Bruck--Ryser--Chowla theorem \cite{bruck1949nonexistence,chowla1950combinatorial} provides necessary conditions (when $r \equiv 1 \text{ or } 2$ (mod $4$)), ruling out certain values of $r$ (for example, $r=6$).

However, a complete characterization of the existence of $S(2,r,n)$ for arbitrary parameters $(r,n)$ is not known.

\section{Our Results}

\noindent
We begin by determining the exact extremal value for linear stars.
\begin{proposition}\label{starprop}
For $r \geq 3$, $ex_r^{\mathrm{lin}}(n , S_k^r) \leq \dfrac{n(k-1)}{r}$. Equality holds if and only if the linear $r$-uniform hypergraph is $(k-1)$-regular, given it exists.    
\end{proposition}

To complement this, we obtain a matching lower bound for general linear hypertrees under mild divisibility and existence assumptions, via block constructions arising from Steiner systems.

\begin{theorem}\label{lowerbound}
    Let $r\ge 3$ and $k\ge 2$. Let $t=(r-1)(k-1)+1$.
Assume that $t\mid n$ and that the Steiner system $S(2,r,t)$ exists.
Then,
$ex^{\mathrm{lin}}_r(n, T^r_k)\ \ge\ \frac{n(k-1)}{r}$.
This bound is sharp when $T_k^r$ is $S_k^r$.
\end{theorem}

Next, we turn to hypertrees on four hyperedges. For the configuration $B_4^r$, we obtain a sharp extremal characterization.

\begin{theorem}\label{b4thm}
    For $r \geq 3$, $ex_r^{\mathrm{lin}}(n,{B_4^r}) \leq \dfrac{(r+1)n}{r}$. Equality holds if and only if the linear $r$-uniform hypergraph is the union of disjoint Steiner systems $S(2, r, r^2)$, given that the Steiner system $S(2,r,r^2)$ exists.
    \end{theorem}

For the crown hypergraph $E_4^r$, we obtain the following upper bound.

\begin{theorem}\label{E4thm}
        For $r \geq 3$, $ex_r^{\mathrm{lin}}(n,E_4^r) \leq \dfrac{(2r-1)n}{r}$.
    \end{theorem}

\noindent  
The linear Tur\'{a}n number for $S_4^r$ follows directly from Proposition~\ref{starprop}.  
The upper bound for $P_4^r$ follows from the bound on $P_k^r$ in \cite{zhou2025turan}, and we obtain a lower bound for $P_4^r$ under certain divisibility and existence assumptions.
\begin{theorem}\label{P4thm}
    Assume $r^2\mid n$ and that the Steiner system $S(2,r,r^2)$ exists. Then, $ex_r^{\mathrm{lin}}(n,P_4^r) \geq \dfrac{(r+1)n}{r}$.
\end{theorem}

    \section{Proof of \Cref{starprop} and \Cref{lowerbound}}
    \textit{Proof of \Cref{starprop}.}
        If $H$ is $S_k^r$-free, then $d(v) \leq k-1$ for all $v \in V(H)$. Since each edge consists of $r$ vertices, we get an analogue of the handshaking lemma as $\sum_{v \in V(H)}d(v) = r|E(H)|$. Thus,
        $r|E(H)| \leq n(k-1) \implies |E(H)| \leq \dfrac{n(k-1)}{r}$.
        Therefore we get $ex_r^{\mathrm{lin}}(n , S_k^r) \leq \dfrac{n(k-1)}{r}$ and clearly for equality we need $d(v) = k-1$ for all $v \in V(H)$. Note that a linear $r$-uniform $(k-1)$-regular hypergraph may not exist. In such cases, equality is not feasible. Clearly, $r \mid n(k-1)$ is a necessary condition. Also it is necessary to have,
        \[|E(H)|{r \choose 2} \leq {n\choose 2} \implies (k-1)(r-1) \leq (n-1)\]
        These conditions are not sufficient. For example, when $(n, k, r) = (10, 4, 3)$, 
both the above conditions are satisfied, yet no such hypergraph 
exists. In fact, there is no known general characterization for the existence 
of regular linear $r$-uniform hypergraphs.

    Now we give a constructive proof for \Cref{lowerbound} under the divisibility and existence assumptions.
    \vspace{2mm}
    
    \noindent\textit{Proof of \Cref{lowerbound}.}
Let $n = qt$ and partition the $n$-vertex set into $q$ disjoint parts
$V(H) = V_1 \sqcup V_2 \sqcup \cdots \sqcup V_q$,
such that $|V_i|=t$.
Suppose each $V_i$ induces $H_i$, a component of $H$, and each $H_i$ is a copy of $S(2,r,t)$.

Since each $H_i$ is linear and $r$-uniform by definition, and edges of disjoint components are disjoint, we have $H$ to be linear and $r$-uniform.

By \Cref{steinercount}, we have
\[|E(H)| = \sum_{i=1}^q |E(H_i)|
= \frac{n}{t}\cdot \frac{t(t-1)}{r(r-1)}
= \frac{n(t-1)}{r(r-1)} = \frac{n(r-1)(k-1)}{r(r-1)} = \frac{n(k-1)}{r}\]
\begin{claim}
    Every linear $r$-uniform hypertree with $k$ edges has exactly $(r-1)k+1$ vertices.
\end{claim}
\begin{proof}
Starting with one edge, clearly it contains $r$ vertices.
When we add a new edge in a linear $r$-uniform hypertree, it intersects the existing vertex set in exactly one
vertex, hence contributes precisely $r-1$ new vertices.
After adding $k-1$ further edges, the total vertex count is
$r + (k-1)(r-1) = (r-1)k + 1$.
\qed\end{proof}
Suppose there exists a $T^r_k$ in $H$, clearly it must be in some $H_i$. From the above claim $|V(H_i)| \geq (r-1)k +1 = \bigl((r-1)(k-1)+1\bigr) + (r-1) = t + (r-1)$. Thus we get a contradiction.

Therefore, $H$ is a $T^r_k$-free graph and $|E(H)|  = \frac{n(k-1)}{r}$. Thus, under the given divisibility and existence assumptions we have $ex_r^{\mathrm{lin}}(n,T_k^r) \geq \frac{n(k-1)}{r}$. From \cref{starprop} it is clear that the bound is sharp when $T^r_k \cong S^r_k$.  
    \section{Proof of \Cref{b4thm}}
    \subsection{Proof of inequality}\label{ineqb4}
        Suppose the statement in \Cref{b4thm} is false. Assume $H$ to be a minimal counterexample. That is, $H$ is a $B_4^r$-free linear $r$-uniform  hypergraph on $n$ vertices and   
        \[|E(H)| > \frac{(r+1)n}{r}\]
        and for any subhypergraph $H'$ of $H$, \Cref{b4thm} holds for $H'$.
        \begin{claim}
            $\delta(H) \geq 2$.        \end{claim}
            \begin{proof}
                If $H$ contains an isolated vertex, then we can drop that vertex, resulting in a smaller counterexample. Thus, suppose there exists $v \in V(H)$, such that $d(v) =1$. Let $H' = H \setminus \{v\}$. Clearly, \[|E(H')| = |E(H)| -1 > \frac{(r+1)n}{r} -1 > \frac{(r+1)n}{r} - \frac{(r+1)}{r} = \frac{(r+1)(n-1)}{r}\]
                Thus $H'$ is a counterexample for \Cref{b4thm}, contradicting the minimality of $H$.
            \qed\end{proof}

                    $H$ must be connected, otherwise, we will get at least one connected component of $H$ as a counterexample, contradicting minimality of $H$.
        \begin{claim}
            $\Delta(H) \geq (r+2)$.
            \begin{proof}
                Suppose $\Delta(H) \leq (r+1)$, then $d(v) \leq (r+1)$ for all $v \in V(H)$. Therefore, $\sum _{v \in V(H)}d(v) \leq n(r+1)$. Since $H$ is $r$-uniform, we have 
                \[\sum_{v\in V(H)}d(v) = r|E(H)| > (r+1)n\]
                Thus we get a contradiction.
            \qed\end{proof}
        \end{claim}
        Let $v \in V(H)$ be such that $d(v) = \Delta(H) = k$. Thus we get an expanded star $S=S_k^r$ centered at $v$. Let the edges of $S$ be $E(S)=\{e_1,e_2.\dots,e_k\}$. Let $u \in e_1\setminus \{v\}$. Since $\delta(H) \geq 2$, there exists $f \in E(H)$ containing $u$ such that $f \neq e_1$. Suppose $f$ intersects $t$ many hyperedges from $E(S)$, without loss of generality let these edges be $\{e_1,e_2,\dots,e_t\}$. Since $H$ is $r$-uniform, $t \leq r$.
        
        Note that, $k \geq r+2$, therefore $t \leq k-2$ and thus $e_{t+1},e_{t+2} \in E(S)$. It is easy to see that the edges $f,e_1,e_{t+1},e_{t+2}$ form a $B_4^r$. Thus we get a contradiction. Therefore, no such counterexample exists.
    \subsection{Characterizing extremal Hypergraphs}
        For characterizing the extremal hypergraphs, we will only focus on the cases when $S(2,r,r^2)$ exists. If $S(2,r,r^2)$ does not exist, then the bound in \Cref{b4thm} is not tight.
    \newline\textbf{$\bullet$ If Part}
    
\noindent Assume $H$ to be a disjoint union of copies of a Steiner system $S(2,r,r^2)$.
Fix one component $C\cong S(2,r,r^2)$ with $|V(C)|=r^2$.
\begin{claim}
    $|E(C)| = r(r+1)$.
\end{claim}
\begin{proof}
    From \Cref{steinercount}, we have \[|E(C)| = \frac{\binom{r^2}{2}}{\binom{r}{2}} = \frac{r^2(r^2-1)}{r(r-1)} = r(r+1) \tag*{\qed}\]
\end{proof}
If $H$ has $t$ components, then $n=t r^2$ and $|E(H)| = t\cdot r(r+1)=\frac{(r+1)n}{r}$.

\begin{lemma}\label{lem:steiner}
    If $C \cong S(2,r,r^2)$, and $e \in E(C)$, then for any $v\in V(C)$ such that $v \notin e$, there exists a unique $f \in E(H)$ containing $v$ such that $f\cap e = \emptyset$.
\end{lemma}
\begin{proof}
Fix an edge $e\in E(C)$ and a vertex $v\in V(C)\setminus e$. From \Cref{steinercount} we have $d(v) = r+1$. Hence there are exactly $r+1$ edges of $C$ containing $v$.

For each vertex $u\in e$, the pair $\{u,v\}$ is contained in a unique edge, which we denote by $e_u$. Clearly, if $u\neq u'$, then $e_u\neq e_{u'}$. Hence, the $r$ vertices of $e$ give rise to $r$ distinct edges containing $v$, each intersecting $e$ in exactly one vertex.

Since $v$ lies in exactly $r+1$ edges in total, there exists precisely one further edge $f$ containing $v$ that is different from all the $e_u$. This edge $f$ cannot intersect $e$; otherwise it would coincide with $e_u$ for some $u\in e$. Therefore, $f\cap e=\emptyset$.
\qed\end{proof}

\begin{claim}
    $C$ is $B_4^r$-free.
\end{claim}
\begin{proof}
    Suppose there exists a $B_4^r$ in $C$ and assume it consists of edges $e_1,e_2,e_3$ and $e_4$, where $e_1\cap e_2\cap e_3 = \{v\}$ and $e_3\cap e_4 \neq \emptyset$. Clearly $v \notin e_4$ and $e_1,e_2$ contain $v$ and are disjoint from $e_4$. Thus we get a contradiction to \Cref{lem:steiner}.
\qed\end{proof}
Thus, $|E(H)| = \dfrac{(r+1)n}{r}$ and $H$ is $B_4^r$-free.

\noindent\textbf{$\bullet$ Only If Part}

\noindent Now let $H$ be a $B_4^r$-free linear $r$-uniform hypergraph on $n$ vertices satisfying $|E(H)|=\dfrac{(r+1)n}{r}$. Without loss of generality, assume $H$ is connected. We need to show that $H$ is $S(2,r,r^2)$.
\begin{claim}
    $H$ is $(r+1)$-regular
\end{claim}
\begin{proof}
    Since $H$ is $r$-uniform, $\sum_{v\in V(H)} d(v) = r|E(H)| = (r+1)n$. Suppose there exists a vertex $v \in V(H)$ such that $d(v) = k \geq r+2$. Now consider the $S_k^r$ centered at $v$. It is easy to see that there exists a vertex in $S_k^r$, other than $v$, with degree at least $2$, otherwise $H \cong S_k^r$, which will result, $n = k(r-1) +1$ and $|E(H)| = k = \frac{n-1}{r-1} < \frac{(r+1)n}{r}$, a contradiction.
    
    Thus, using the arguments as in the proof of inequality in \cref{ineqb4}, we will get a $B_4^r$ in $H$.

    Thus every vertex has degree at most $r+1$ and the average is $r+1$, forcing
$d(v)=r+1$ for all $v\in V(H)$, i.e.\ $H$ is $(r+1)$-regular.
\qed\end{proof}

Let $v \in V(H)$ and the $r+1$ edges containing $v$ be $e_1,\dots,e_{r+1}$.
By linearity, these edges are pairwise disjoint outside $v$, so the set
\[
X:=\{v\}\ \cup\ \bigcup_{i=1}^{r+1}(e_i\setminus\{v\})
\]
has size, $|X| = 1 + (r+1)(r-1) = r^2$.

\begin{claim}
$V(H)=X$.
\end{claim}

\begin{proof}
Suppose for contradiction $ V(H)\setminus X \neq \emptyset$.
Since $H$ is connected, there is an edge $f\in E(H)$ meeting $X$ and containing at least one
vertex outside $X$. Choose such an $f$.
Clearly, $v\notin f$.
Let $u\in f\cap X$ (necessarily $u\neq v$). Without loss of generality assume $u \in e_1$.

Now $f$ has size $r$, and contains at least one vertex outside $X$.
So among the remaining vertices of $f$, at most $r-1$ lie in $X$.
Each vertex of $f \cap \{X\setminus\{v\}\}$ lies in \emph{exactly one} of the star edges
$e_1,e_2,\dots,e_{r+1}$, so $f$ can intersect at most $r-1$ of these $r+1$ edges
(including $e_1$). Therefore there exist \emph{two} distinct edges,
say $e_a$ and $e_b$ among $\{e_2,\dots,e_{r+1}\}$, that are disjoint from $f$.

But then the hyperedges, $f,e_1,e_a,e_b$ form a copy of $B_4^r$, contradicting that $H$ is $B_4^r$-free.
\qed\end{proof}

Finally, we show $H$ is a Steiner system $S(2,r,r^2)$.
Because $H$ is linear, any unordered pair of vertices is contained in \emph{at most one} edge.
Therefore, counting vertex-pairs inside edges gives
\[
|E(H)|\binom{r}{2} \le \binom{r^2}{2}.
\]
But substituting $|E(H)|=\dfrac{(r+1)n}{r} = \dfrac{(r+1)r^2}{r}= r(r+1)$ yields equality:
\[
r(r+1)\binom{r}{2}
= r(r+1)\cdot \frac{r(r-1)}{2}
= \frac{r^2(r^2-1)}{2}
= \binom{r^2}{2}.
\]
Hence every pair of vertices in $H$ lies in \emph{exactly one} edge, i.e.\ $H$ is $S(2,r,r^2)$.

\section{Proof of \Cref{E4thm}}
 We begin by outlining the main idea of the proof, and then present the argument in a sequence of steps.

\medskip
\noindent\textbf{Proof idea.}
We argue by contradiction. Assuming the hypergraph has more than $\frac{(2r-1)n}{r}$ edges, we show that the degree distribution forces the existence of an edge whose vertices have large total degree. We then show that such an edge must contain vertices with sufficiently large degrees, which yields a copy of $E_4^r$ with that edge as its base, a contradiction.

\medskip

\noindent\textbf{Step 1:} \textit{Minimal counterexample and basic properties.}
\newline Suppose the statement in \Cref{E4thm} is false. Let $H$ be a minimal counterexample. Thus, $H$ is an $E_4^r$-free linear $r$-uniform hypergraph on $n$ vertices with $|E(H)| > \frac{(2r-1)n}{r}$, and every proper subhypergraph satisfies the bound.

Since $H$ is minimal, we may assume that $\delta(H)\ge 2$. Indeed, if there exists an isolated vertex $v \in V(H)$, we may remove it without affecting the number of edges. Similarly, if there exists a vertex $v \in V(H)$ with $d(v)=1$, then for $H' = H \setminus \{v\}$,
\[
|E(H')| = |E(H)| -1 > \frac{n(2r-1)}{r}-1 > \frac{(n-1)(2r-1)}{r},
\]
contradicting minimality. Hence $\delta(H)\ge 2$.

\medskip

\noindent\textbf{Step 2:} \textit{A local condition forcing a crown.}
\newline We will now identify a configuration that guarantees the presence of $E_4^r$.

\begin{lemma}\label{crownlem}
Let $H$ be a linear $r$-uniform hypergraph. If there exists an edge $e\in E(H)$ containing three vertices $a,b,c$ with
$d(a)\ge 2r$, $d(b)\ge r+1$, and $d(c)\ge 2$, then $H$ contains a copy of $E_4^r$ with base $e$.
\end{lemma}

\begin{proof}
Since $d(c)\ge 2$, pick an edge $f\neq e$ containing $c$. By linearity, any edge through $b$ intersects $f$ in at most one of the $r-1$ vertices in $f\setminus\{c\}$. Thus at most $r-1$ such edges intersect $f$, and since $d(b)-1\ge r$, there exists $g\neq e$ through $b$ disjoint from $f$.

Similarly, an edge $h$ through $a$ can intersect $f\cup g$ in at most $2(r-1)$ vertices, but $d(a)-1\ge 2r-1>2(r-1)$, so there exists $h\neq e$ disjoint from both $f$ and $g$.

Thus, $f,g,h$ are pairwise disjoint edges intersecting $e$, giving a copy of $E_4^r$ with base $e$.
\qed
\end{proof}

\noindent\emph{Strategy.}
We will show that such an edge $e$ must exist by a global counting argument.

\medskip
\noindent\textbf{Step 3:} \textit{Measuring edge density via degrees.}
\newline For each edge $e \in E(H)$, define
$s(e) = \sum_{v \in e} d(v).$

\medskip
\noindent\emph{Idea.}
The quantity $s(e)$ measures how \textit{dense} the neighbourhood of $e$ is. Our goal is to show that some edge has very large $s(e)$.

It is easy to observe that
\[
\sum_{e \in E(H)} s(e) = \sum_{e \in E(H)}\left(\sum_{v \in e} d(v)\right) = \sum_{v \in V(H)}\left(\sum_{e \ni v}d(v)\right) =\sum_{v \in V(H)} d(v)^2.
\]

\noindent\textbf{Step 4:} \textit{Controlling large-degree vertices.}
\newline  Set $A = 4r-3$ and $B = r^2+3r-3$. The vertices with degree at least $A$, are called \textit{large vertices}.
Let $L(H) = \{v \in V(H): d(v)\ge A\}$ be the set of large vertices and define
\[
s^*(e)=
\begin{cases}
\min\{s(e),B\}, & \text{if $e$ contains a large vertex},\\
s(e), & \text{otherwise}.
\end{cases}
\]

\medskip
\noindent\emph{Why truncation?}
Edges containing very large-degree vertices can artificially inflate $s(e)$. The truncation $s^*(e)$ controls this effect.

\begin{lemma}\label{onelarge}
If $v\in L(H)$ and $e$ contains $v$, then for all $u\in e\setminus\{v\}$, $d(u)<r+1$.
\end{lemma}

\begin{proof}
        Suppose there exists $u \in e \setminus \{v\}$ such that $d(u) \geq r+1$. Note that $d(v) \geq A \geq 2r$. Let $c \in e\setminus \{u,v\}$. Clearly $d(c) \geq 2$, since $\delta(H) \geq 2$. Thus from \Cref{crownlem} we have a $E_4^r$ in $H$.
\qed
\end{proof}

\medskip
\noindent\textbf{Step 5:} \textit{Bounding truncation error.} \newline Let $E_v(H)$ denote the set of hyperedges in $E(H)$ containing the vertex $v$.
\begin{lemma}\label{lem:diff}
If $v\in L(H)$ and $d(v)=d$, then
$\sum_{e\in E_v(H)} (s(e)-s^*(e)) \le d^2-Ad.$
\end{lemma}
\begin{proof}
From \Cref{onelarge}, for each $e\in E_v(H)$, $s(e)\le d+r(r-1)$.
Clearly, for $e \in E_v(H)$, $s^*(e)=\min\{s(e),B\}$. Thus,
\[
s(e)-s^*(e)\le \max\{0,\,s(e)-B\}\le \max\{0,(d+r(r-1))-B\}=\max\{0,d-A\}.
\]
Since $d\ge A$, the maximum equals $d-A$. There are $|E_v(H)|=d$ edges containing $v$, so $\sum_{e\in E_v(H)} (s(e)-s^*(e))\le d(d-A)=d^2-Ad.$ \qed \end{proof}    

This bounds how much truncation reduces the contribution from edges containing large vertices.

\medskip
\noindent\textbf{Step 6:} \textit{Balancing the degree sequence.}
\newline Since $\sum_v d(v)=r|E(H)| > r\frac{n(2r-1)}{r} = n(2r-1)$, and $H$ is minimal, we have $r|E(H)|=n(2r-1)+l$ for some $l\in[r]$, we compare the degree sequence with a near-uniform baseline.

\begin{lemma}\label{lem:transfer}
There exists a sequence of functions $f_0,f_1,\dots,f_k:V(H)\to \mathbb{N}$ and a partition
$V(H)=I\sqcup D$ such that:
\begin{enumerate}
\item For all $v\in V(H)$, $2r-1 \leq f_0(v)\leq 3r-1$ and $\sum_{v} f_0(v)=(2r-1)n+l$;
\item $f_k(v)=d(v)$ for all $v \in V(H)$;
\item For $ i \in [k-1]$, there exist $x_i\in I$, $y_i\in D$ with
$f_i(x_i)=f_{i-1}(x_i)+1$, $f_i(y_i)=f_{i-1}(y_i)-1$, and all other values unchanged;
\item if $v\in D$, then $f_0(v)=2r-1$ 
\end{enumerate}
\end{lemma}
\noindent\emph{Idea.}
We start with an almost regular degree sequence and gradually \textit{unbalance} it to match the actual sequence. Because a sum of squares is minimized when its components are as equal as possible, tracking the exact increase during this process provides a rigorous baseline to lower bound the total edge density via $\sum_v d(v)^2$.
\begin{proof}
Label the vertices in $V(H)$ as $v_1$ through $v_n$ so that $(d(v_i))^n_{i=1}$ is in non-decreasing order. We follow the following algorithm to define $f_0$.
\begin{enumerate}
    \item First set $f_0(v) = 2r-1$ for all $v \in V(H)$.
    \item Set  $R= l$ and $i =n$. 
    \item While $R > 0$,
    \begin{itemize}
            \item Update $f_0(v_i) = min\{(2r-1 + R), d(v_i)\}$ and $R = R - (f_0(v_i) - 2r+1)$ and $i = i-1$. (Note that update to $f_0(v_i)$ happens only when $d(v_i) > f_0(v_i)$)
    \end{itemize}
\end{enumerate}
Clearly we get $2r-1 \leq f_0(v) \leq 2r-1+ l \leq 3r-1$ and $\sum_v f_0(v) = (2r-1)n +l$.

Now we iteratively transform $f_i$ into the degree function $d$ while preserving the total sum. We define $f_i$ for $i >0$, assuming $f_{i-1}$ is already defined. Take the minimal $a \in [n]$ such that $f_{i-1}(v_a) > d(v_a)$ and the maximal $b \in [n]$ such that $f_{i-1}(v_b) < d(v_b)$. If no such $a$ and $b$ exist, then $f_{i-1}=d$, thus take $i-1 = k$. Otherwise, define $f_i(v_a) = f_{i-1}(v_a) - 1,f_i(v_b) = f_{i-1}(v_b) + 1$, and $f_i(v_c) = f_{i-1}(v_c)$ for $c \notin \{a, b\}$. This strictly decreases
$\sum_v |f(v)-d(v)|$, so the process terminates after finitely many steps at $f_k=d$.

Let $D$ be the set of vertices that are ever decreased and $I$ the rest. Any vertex in $D$ cannot
have been among the initially increased vertices (while defining $f_0$), since it ends up requiring a decrease.
Hence $f_0(v)=2r-1$ for all $v\in D$.
\qed\end{proof}

\begin{definition}
Consider a sequence $\{f_i\}_{i=0}^k$ of functions as in \Cref{lem:transfer}.
\begin{itemize}
    \item For $0\leq i \leq k$, define $T_i =\sum_{v}f_i(v)^2$.
    \item For $i \in [k]$, define $\Delta_i = T_i - T_{i-1}$.
    \item For $v \in V(H)$, let $I_v = \{ i \in [k] : f_i(v) \neq f_{i-1}(v)\}$.
    \item For $v \in V(H)$, define $\Delta_v = \sum_{i \in I_v} \Delta_i$.
\end{itemize}    
\end{definition}

\medskip
\noindent\emph{Interpretation.}
The quantity $T_i$ captures how concentrated the degree sequence is at stage $i$. The increments $\Delta_i$ track how this concentration changes, and $\Delta_v$ measures the total contribution of a vertex $v$ to this change.

\begin{lemma}\label{lem:delta}
Let $v\in L(H)$ with $d(v)=d$. Then
$\Delta_v\;\ge\; d^2-Ad + (r-2)(3r-1)$.
\end{lemma}
\begin{proof}
If $v\in L(H)$, then $v\in I$, so $v$ is never decreased. Let $x = f_0(v)$, thus $2r-1\leq x \leq 3r-1$.
Each time $f_{i-1}(v)$ increases from $t-1$ to $t$ (in $f_i(v)$), the square increases by $t^2-(t-1)^2=2t-1$. Summing over the $(d-x)$ increments from $x$ to $d$ gives $\sum_{t=x+1}^{d} (2t-1) = d^2-x^2$.

Each such increment step is paired with a decrease of some $y\in D$ (since the sum remains unchanged).
Vertices in $D$ start at $2r-1$ and are never increased, so the maximum possible square-loss
in any single decrease is $(2r-1)^2-(2r-2)^2=4r-3=A$.
Hence, the total paired loss over the $(d-x)$ increments is at most $A(d-x)$. Therefore, \begin{equation}\label{lemma6eqn1}
    \Delta_v \ge (d^2-x^2)-A(d-x)=d^2-Ad+(Ax-x^2)
\end{equation}
It is easy to check that the term $Ax-x^2=x(A-x)$ is minimized on $[2r-1,3r-1]$ at $x=3r-1$, giving 
\begin{equation}\label{lemma6eqn2}
    Ax-x^2 \ge (3r-1)\bigl((4r-3)-(3r-1)\bigr)=(3r-1)(r-2)
\end{equation} 
Thus, from \cref{lemma6eqn1,lemma6eqn2} we get the required bound.
\qed\end{proof}

\noindent\textbf{Step 7:} \textit{Global bound.}
\newline Let $T^*(H)=\sum_{e\in E(H)} s^*(e)$.

\begin{lemma}\label{lem:T*bound}
If $H$ is a minimal counterexample on $n$ vertices, then
\[
T^*(H) \geq n(2r-1)^2  + l(4r-1) +(r-2)(3r-1)|L(H)|
\]
\end{lemma}
\begin{proof}
From \Cref{lem:transfer}, $T_k = \sum_{v}f_k(v)^2 = \sum_vd(v)^2=\sum_es(e)$.
Thus,
\begin{equation}\label{eqn1}
    T^*(H) = \sum_{e\in E(H)}s(e) - \sum_{e\in E(H)}(s(e) - s^*(e)) = T_k - \sum_{e\in E(H)}(s(e)-s^*(e))
\end{equation}
If an edge $e$ does not contain a large vertex, then by definition $s^*(e) = s(e)$. Thus the sum can be restricted to the edges containing at least one large vertex, but from \Cref{onelarge} we know that no edge can contain two large vertices. Thus,
\begin{equation}\label{eqn2}
    \sum_{e \in E(H)}(s(e) - s^*(e)) = \sum_{v\in L(H)}\sum_{e\in E_v(H)}(s(e)-s^*(e))
\end{equation}
Since, for $i \in [k]$, $\Delta_i = T_i - T_{i-1}$, by telescoping we get, $T_k = T_0 + \sum_{i =1}^k\Delta_i = T_0 + \sum_{v\in I}\Delta_v \geq T_0 + \sum_{v \in L(H)}\Delta_v$. Thus from \cref{eqn1,eqn2} we get,
\begin{equation}\label{eqn3}
    T^*(H) \geq T_0 + \sum_{v \in L(H)}\Delta_v - \sum_{v\in L(H)}\sum_{e\in E_v(H)}(s(e)-s^*(e))
\end{equation}
Note that $\sum_v f_0(v) = (2r-1)n + l$. Clearly, $T_0$ is minimized when for all $v \in V(H)$, $f_0(v)$ is as equal as possible. Therefore, \[T_0 \geq (n-l)(2r-1)^2 + l(2r)^2 = n(2r-1)^2 + l(4r-1)\]
Let $d(v) =d$, thus applying \Cref{lem:diff,lem:transfer,lem:delta} to \cref{eqn3} we get,
\begin{align*}
    T^*(H)&\geq T_0 + |L(H)|(d^2-Ad+(r-2)(3r-1)) - |L(H)|(d^2-Ad)\\
    &\geq  n(2r-1)^2 + l(4r-1) + (r-2)(3r-1)|L(H)|. \tag*{\qed}
\end{align*}
\end{proof}

\medskip

This combines all previous estimates to obtain a strong lower bound on the total truncated edge contribution.

\medskip
\noindent\textbf{Step 8:} \textit{Extracting a dense edge and resulting a contradiction.}
\newline Since $|E(H)| = \frac{n(2r-1)+l}{r}$ where $l > 0$. Thus from \Cref{lem:T*bound},
\begin{equation}\label{eqn4}
    \frac{T^*(H)}{|E(H)|} \geq \frac{(2r-1)^2n + l(4r-1)+(r-2)(3r-1)|L(H)|}{\frac{n(2r-1) + l}{r}}
\end{equation}
From \cref{eqn4} we get,
\begin{equation*}
     \frac{T^*(H)}{|E(H)|} > \frac{(2r-1)^2n + l(2r-1)}{\frac{n(2r-1) + l}{r}}= r(2r-1)
\end{equation*}
Since $T^*(H) = \sum_{e\in E(H)}s^*(e)$, we get that there exists an $e \in E(H)$ such that $s^*(e) \geq r(2r-1)+1$. Note that for $r \geq 3$, $(2r^2 - r +1) - B = (2r^2-r+1) - (r^2 + 3r - 3) = (r-2)^2 > 0$. Thus, $s^*(e) > B$ and therefore, $s^*(e) = s(e) =\sum_{v \in e}d(v)$. Since, $s(e) > B$, $e$ does not contain any large vertex.

By averaging argument, there exists $x \in e$ such that $d(x) \geq 2r$. Suppose for all $y \in e \setminus \{x\}$, $d(y) \leq r$. Since $e$ has no large vertex, $d(x) \leq 4r - 4$. Thus, $s(e) \leq (4r-4) + (r-1)r = r^2 +3r -4$. But $s(e) =s^*(e) \geq 2r^2 -r +1$. Thus, 
\begin{equation*}
    r^2 +3r -4 \geq 2r^2 -r +1
    \implies 0\geq r^2-4r+5 = (r-2)^2+1
\end{equation*}
Thus, we get a contradiction. Therefore, there exists $y \in e\setminus \{x\}$ such that $d(y) \geq r+1$. Now since $\delta(H) \geq 2$, $d(z) \geq  2$ for $z \in e \setminus\{x,y\}$. Thus, by \Cref{crownlem}, $H$ contains a crown, which is a contradiction. 

\section{Proof of \Cref{P4thm}}

Let $n = qr^2$ and partition the $n$-vertex set into $q$ disjoint parts
$V(H) = V_1 \sqcup V_2 \sqcup \cdots \sqcup V_q$,
such that $|V_i|=r^2$.
Suppose each $V_i$ induces $H_i$, a component of $H$, and each $H_i$ is a copy of $S(2,r,r^2)$.

Since each $H_i$ is linear and $r$-uniform by definition, and edges of disjoint components are disjoint, we have $H$ to be linear and $r$-uniform.

By \Cref{steinercount}, we have
\[|E(H)| = \sum_{i=1}^q |E(H_i)|
= \frac{n}{r^2}\cdot \frac{r^2(r^2-1)}{r(r-1)}
= \frac{(r+1)n}{r}.\]
Suppose there exists a $P^r_4$ in $H$, clearly it must be in some $H_i$. Let $e_1,e_2, e_3$ and $e_4 $ be the edges of $P^r_4$, such that $e_1 \cap e_3= \emptyset$, $e_1 \cap e_4= \emptyset$ and $v\in e_3\cap e_4$. Note that $v\notin e_1$ and we get two edges $e_3$ and $e_4$ containing the vertex $v$ that are disjoint from $e_1$. This contradicts \Cref{lem:steiner}.

Therefore, $H$ is a $P^r_4$-free graph with $|E(H)|  = \frac{(r+1)n}{r}$. Thus, we have $ex_r^{\mathrm{lin}}(n,P_4^r) \geq \frac{(r+1)n}{r}$. It is easy to check that for $q=1$, the bound is strict, since $S(2,r,r^2)$ maximizes the number of edges.

\section*{Conclusion and Future Directions}

In this paper, we studied linear Tur\'{a}n numbers for $r$-uniform linear hypertrees. 
We provided a general construction giving the lower bound $ex_r^{\mathrm{lin}}(n,T_k^r)\ge \frac{n(k-1)}{r}$ under mild assumptions, and showed that this bound is sharp for stars. 
For hypertrees on four edges, we determined the exact extremal value for $B_4^r$, proved an upper bound for the crown $E_4^r$, and established a lower bound for $P_4^r$.

We conclude by proposing a conjectural sharp upper bound for $P_4^r$, matching the lower bound we provided in \Cref{P4thm}.
\begin{conjecture}
    \emph{$ex_r^{\mathrm{lin}}(n,P_4^r) \leq \dfrac{(r+1)n}{r}$. Equality holds if and only if the linear $r$-uniform hypergraph is the union of disjoint Steiner systems $S(2, r, r^2)$, given that the Steiner system $S(2,r,r^2)$ exists.}
\end{conjecture}
A natural direction for future work is to determine a sharp upper bound for $E_4^r$.

Another direction that can be of interest is to obtain sharp bounds for $P_k^r$, as the known bound,
$ex_r^{\mathrm{lin}}(n,P_k^{r}) \le \frac{(2r-3)kn}{2}$ from~\cite{zhou2025turan}, can be far from optimal, and we already have sharp bound for paths in the graph setting (check \cite{gallai1959maximal}).

\bibliographystyle{plain}
\bibliography{references}

\end{document}